\documentclass[12pt]{amsart}
\usepackage{times, pdfpages, graphicx}
\usepackage{comment}
\usepackage{cancel}
\usepackage{xcolor}

\usepackage{exscale}
\usepackage{epsfig}

\usepackage{amsthm}
\usepackage{amsfonts}

\usepackage{graphics,psfrag,graphicx,subfigure,epsfig, xypic}
\usepackage{amsmath,amsfonts,amssymb,amstext,amsthm,amscd,mathrsfs}
\usepackage[english]{babel}
\usepackage[all]{xy}
\usepackage{multicol}
\usepackage{rotating}

\usepackage{amsmath, amsfonts, amssymb, amscd, enumerate}
\usepackage{color}
\usepackage[all]{xy}

\theoremstyle{plain}
 \newtheorem{theo}{Theorem}[section]

\theoremstyle{plain}

\newtheorem{proposition}[theo]{Proposition}
\newtheorem{definition}[theo]{Definition}
\newtheorem{remark}[theo]{Remark}

\newcommand{\beq}{\begin{equation}}
\newcommand{\eeq}{\end{equation}}

\newcommand{\Arg}{\textnormal{Arg}}

\newcommand{\Oc}{\mathbb O}

\newcommand{\C}{\mathbb{C}}

\newcommand{\R}{\mathbb{R}}
\renewcommand{\H}{\mathbb{H}}
\newcommand{\Z}{\mathbb{Z}}

\newcommand{\bK}{\mathbb{K}}

\newcommand{\bS}{\mathbb{S}}

\newcommand{\scrE}{\mathscr E}

\newcommand{\SK}{\mathbb S_{\mathbb K}}

\newcommand{\ra}{\rightarrow}

\renewcommand{\square}{\kern1pt\vbox
{\hrule height 0.6pt\hbox{\vrule width 0.6pt\hskip 3pt
\vbox{\vskip 6pt}\hskip 3pt\vrule width 0.6pt}\hrule height0.6pt}\kern1pt}

\renewcommand\Re{\operatorname{Re}}
\renewcommand\Im{\operatorname{Im}}

\renewcommand{\Re}{{\rm Re}}
\renewcommand{\Im}{{\rm Im}}

\newcommand{\be}{\begin{equation}}
\newcommand{\ee}{\end{equation}}

\def\<#1,#2>{\langle\,#1,\,#2\,\rangle}
\newcommand{\arr}{\begin{array}{rlll}}
\newcommand{\ea}{\end{array}}
\newcommand{\bea}{\begin{eqnarray}}
\newcommand{\eea}{\end{eqnarray}}
\newcommand{\bean}{\begin{eqnarray*}}
\newcommand{\eean}{\end{eqnarray*}}

\def\sideremark#1{\ifvmode\leavevmode\fi\vadjust{
\vbox to0pt{\hbox to 0pt{\hskip\hsize\hskip1em
\vbox{\hsize3cm\tiny\raggedright\pretolerance10000
\noindent #1\hfill}\hss}\vbox to8pt{\vfil}\vss}}}

\newcounter{ssig}
\newcounter{ttig}
\begin{document}

\title[Slice conformality and Riemann manifolds...]{Slice conformality and Riemann manifolds on quaternions and octonions}

\author{Graziano Gentili}
\address{DiMaI, Universit\`a di Firenze, Viale Morgagni 67/A,\ Firenze, Italy}
\email { graziano.gentili@unifi.it }
\author{Jasna Prezelj}
\address{Fakulteta za matematiko in fiziko Jadranska 19 1000
  Ljubljana, Slovenija,
  UP FAMNIT, Glagolja\v ska 8, Koper Slovenija,
IMFM Jadranska 19 1000
  Ljubljana, Slovenija
}

\email { jasna.prezelj@fmf.uni-lj.si}
\author {Fabio  Vlacci}\address{DiSPeS Universit\`a di Trieste Piazzale Europa 1,\ Trieste,
  Italy} \email{ fvlacci@units.it}
\thanks{\rm
  The first and third authors were partly supported by INdAM, through: GNSAGA; INdAM project ``Hypercomplex function theory and applications''.
  The first author was also partially supported by MIUR, through the projects: Finanziamento Premiale FOE 2014 ``Splines for accUrate NumeRics: adaptIve models for Simulation Environments''.
  The second author
  was partially supported by research program P1-0291 and by research
  projects J1-7256 and J1-9104 at Slovenian Research Agency.
The third author was also partially supported by PRIN 2017 ``Real and complex manifolds: topology, geometry and holomorphic dynamics'.
}

\makeatletter
\@namedef{subjclassname@2020}{\textup{2020} Mathematics Subject Classification}
\makeatother
\begin{abstract}

In this paper we establish  quaternionic and octonionic analogs of the classical
Riemann surfaces. The construction
of these manifolds has nice peculiarities and the scrutiny of Bernhard
Riemann approach to Riemann surfaces, mainly based on conformality,
leads to the definition of slice conformal or slice isothermal
parameterization  of  quaternionic or octonionic Riemann manifolds. These
new classes of manifolds include slice regular quaternionic and
octonionic curves, graphs of slice regular functions, the $4$ and $8$
dimensional spheres, the helicoidal and catenoidal $4$ and $8$
dimensional manifolds.
Using appropriate   Riemann manifolds, we also give a unified definition of the quaternionic and octonionic logarithm and $n$-th root function.

\end{abstract}


\maketitle

\section{Preface}
  The initial project originating this paper was
giving a well structured and unifying definition of the logarithm and
$n$-th root functions in the quaternionic and octonionic settings.
To this purpose, our first aim was to
  construct the quaternionic and octonionic analogs of the well known
  Riemann surface of the complex logarithm, which in the complex
  setting allows a complete understanding of this function and of its
  branches.

Indeed, the manifolds constructed with this aim revealed new, interesting and peculiar features, so that  they captured the central position among the results of this paper.

We will illustrate how the project of this paper developed
 and, to begin with, point out that for the case of the principal branch of the logarithm,
  definitions were already given in the general setting of Clifford
  Algebras - see, e.g., \cite[Definition 11.24, p. 231]{Gurl} - and
  also specialized to the case of quaternions - see, e.g.,
  \cite[Definition 3.4]{Gen-Vig}.   \\

  Let $\mathbb{K}$ be either the division algebra of quaternions $\H$ or the division algebra of octonions
 $\Oc$; we  denote by $\dim \mathbb{K}$ the {\em real} dimension of $\mathbb{K}$, namely
$\dim \mathbb{H}=4$  and $\dim \mathbb{O}=8$.
 Let $\SK \subset \mathbb K$ be the $2$-sphere or,
 respectively, the $6$-sphere of imaginary units, i.e. the sets of
 $I\in \mathbb K$ such that $I^2=-1$.  For the sake of simplicity,
 both in the case of quaternions and in the case of octonions we will
 simply write $\bS$ instead of $\SK$ since no confusion can arise. The
 construction of the logarithm and its branches given in the complex
 case cannot be directly replicated in the quaternionic and octonionic
 environments. This is mainly due to the fact that the exponential
 function
\[
\exp q=\sum_{n=1}^\infty \dfrac{q^n}{n!}
\]
is an entire function (i.e., its domain of definition is $ \mathbb K$), but cannot be used to define a covering of $\mathbb K\setminus\{0\}$.  In fact, for all $0 \neq x\in \R$, the preimage of $x$ is not a discrete set but consists of infinitely many $2$ or $6$ dimensional spheres.  Indeed, for instance  in the case $x<0$, setting $\bS (2k+1)\pi=\{q(2k+1)\pi : q\in \bS\}$, we have
\[
(\exp)^{-1}(x)=\{ \log|x| + \bS (2k+1)\pi  : k\in \Z\}.
\]
It follows that, contrarily to what happens in the case of the
 complex logarithm, no continuous branch of the quaternionic or
 octonionic logarithm can be defined on any open neighborhood of any
 strictly negative $x\in \R$. A similar phenomenon happens for all
 strictly positive $x\in \R$, except for the principal branch. To overcome this difficulty, we turn our
 attention to the construction of a $4$-dimensional, respectively
 $8$-dimensional, manifold obtained by blowing-up $\mathbb K$ along
 the real axis, and ``adapting'' it to become a domain of definition
 for the quaternionic or octonionic logarithm. Our natural approach to
 perform this construction passes through the recent theory of slice
 regular functions - see, e.g., the monograph \cite{libroGSS} and
 references therein - and leads to the quaternionic and octonionic helicoidal Riemann manifolds
 (which  are  manifolds in the sense of \cite{GGS})
 inspired by the classical helicoidal surface of the space
 $\R^3$.

These manifolds, constructed with
the purpose specified above, have new, interesting and peculiar features that attracted the attention of the authors
and encouraged them to go back to the scrutiny of Bernhard Riemann
approach to holomorphic functions and Riemann surfaces, which was
mainly based on conformality, as in \cite{Riemann}. All this led to a
deeper appreciation of the work of Riemann, to a nice surprise and to
Definition \ref{slice conformality and Riemann manifolds} of
\emph{slice conformal or slice isothermal parameterization} and of
\emph{hypercomplex Riemann manifold}.

 Indeed, the study of slice conformality and the investigation of quaternionic and octonionic Riemann manifolds became the true main subject of this paper.
\vskip .3cm
Let
$\langle \ , \ \rangle$ denote the standard Euclidean scalar product
in $\R^{\dim\bK}\cong \bK$ and, for any purely imaginary unit $I\in \bS$,
set
$$\C_I^\perp=\{q \in \mathbb K : \langle q, x+Iy\rangle =0, \forall (x+Iy) \in \C_I \}$$ to be the orthogonal space to the slice $\C_I=\R+I\R$.
 A $C^1$ injective $\R^N$-valued immersion $f$ defined on a suitable domain $\Omega$ of $\mathbb K$ is called slice conformal or slice isothermal immersion if, for any purely imaginary unit $I\in \mathbb K$ and any $x,y\in \R$, the differential $df(x+Iy)$  is such that both 
 \[
df(x+Iy)_{|\C_I}
\]
and
\[
df(x+Iy)_{|\C_I^\perp}
\]
are conformal. If this is the case, $f(\Omega)$ is called a hypercomplex Riemann
manifold.

The nice surprise was that the quaternionic and octonionic spheres, the helicoidal  and
 catenoidal manifolds, together with the natural quaternionic and octonionic curves, are all hypercomplex Riemann manifolds.

 The study performed in \cite{Orient-preserv} by Ghiloni and        
  Perotti shows that the Jacobian matrix $J_f$ of a slice regular
  function $f$ is such that $\det(J_f)\ge 0$, i.e that $f$ is
  orientation preserving.  Slice conformality is indeed an
  extension of the definition of slice regularity, even in the case of
  $\mathbb K$-valued, orientation preserving immersions defined on a
  domain $\Omega$ of $\mathbb K$: for a fixed non real quaternion $a$,
  the function $f(q) = aq$ in not slice regular, but it is slice
  conformal (actually conformal) and orientation preserving. The following remark
  is basic to help placing  the results of this paper in the right perspective.\\

\noindent{\bf Remark.}
After recalling that the real differential $df$ of a slice regular
function $f:\Omega \to \mathbb K$ is conformal (if non singular) at all real points of
the slice domain $\Omega$ (see, e.g., \cite[Corollary
  8.17.]{libroGSS}), it is worthwhile noticing that to require that
the differential $df$ is conformal at all points of the domain of
definition $\Omega$ may be too restrictive: by a classical result due
to Liouville, for $n>2$ a conformal map from a domain of $\R^n$ to
$\R^n$ is a M\"obius transformation.\\

\noindent In the paper a \emph{standard set of curves} is applied
to study the real
differential of (smooth enough) injective $\R^N$-valued immersions $f$
defined on suitable domains $\Omega$ of $\mathbb K$ (we point out that similar techniques were already introduced in \cite{Orient-preserv, GPSoctonionic, Gurl}). As a result, the
paper can exhibit a collection of quaternionic and octonionic  Riemann manifolds,
inspired by classical Riemann surfaces, which testify the interest of
the approach.


Sub-manifolds of the helicoidal hypercomplex manifolds, endowed with suitable atlases which define different structures, provide a
 natural environment for the definition of the quaternionic and
 octonionic logarithm, and for their possible branches. Once done
 this, the construction of natural manifolds of
 definition for the $n$-th root quaternionic and octonionic
 functions is an easily doable step.

\vskip .2cm
The paper is organized as follows. After a few preliminaries, which also subsume the approach to slice regular functions based on stem functions, Section \ref{sec:2} is dedicated to the definition and construction of classes
 of hypercomplex Riemann manifolds, including quaternionic and octonionic slice regular curves. This construction is based on Theorem \ref{potential_result}, which studies slice conformal curves in terms of their stem functions, and
calls into play the standard set of curves. In Section \ref{4} we present other explicit examples of quaternionic and octonionic regular curves, which comprise the  hypercomplex Riemann sphere, the helicoidal hypercomplex manifold, the
catenoidal hypercomplex manifold and the study of the relations between
them. Section \ref{sec:log} contains the presentation of
natural manifolds for the definition of the quaternionic and octonionic logarithm. The same section contains the
construction of the manifolds of the $n$-th
root quaternionic and octonionic functions.

The authors are grateful and indebted to the anonymous referee for her/his attentive and accurate report, and for the precious comments and suggestions that helped much to put the paper in the present improved form, and to formulate Theorem \ref{potential_result} at the right degree of generality.

\section{Preliminaries}\label{preliminaries}

  As we said, $\mathbb{K}$ denotes either $\H$ or $\Oc$, i.e.,
 the algebras of quaternions or octonions, and
 $\bS\subset \mathbb{K}$  denotes, respectively, the $2$-sphere or $6$-sphere of imaginary units, i.e. the set of
  $I\in \mathbb K$ such that $I^2=-1$.   Given any non real  $q \in
 \mathbb{K},$ there exist (and are uniquely determined) an imaginary
 unit of $\mathbb K$, and two real numbers $x$ and $y>0$, such that
 $q=x+Iy$. With this notation, the conjugate of $q$ will be $\bar q :=
 x-Iy$ and $|q|^2=q\bar q=\bar q q=x^2+y^2$.
 In both cases, each imaginary unit $I$ generates (as a real algebra) a copy
 of the complex plane denoted by $\mathbb{C}_I$. We call such a complex
 plane a {\em slice}.

Let $\Omega$ be a {\em slice domain} of $\mathbb K$, i.e., an open and
connected subset containing real points and such $\Omega_I=\Omega\cap
\C_I$ is a domain of $\C_I$ for all imaginary units $I\in \bS \subset \mathbb K$.  The set of slice regular functions on $\Omega$ is
defined using a family of Cauchy-Riemann operators (see
e.g. \cite{libroGSS, cayley}).

\begin{definition}\label{sliceregular}
 Let $\Omega\subseteq \mathbb K$ be a slice domain and let $f: \Omega \to \mathbb K$ be a function.

 If, for an imaginary unit $I$ of $\mathbb K$, the restriction
  $f_I := f_{|_{\Omega_I}}$
  has continuous partial derivatives and
\begin{equation}\label{sr}
\bar \partial_I f(x+yI) := \frac{1}{2} \left( \frac{\partial}{\partial x} + I \frac{\partial}{\partial y} \right) f_I(x+yI) \equiv 0
\end{equation}
then $f_I$ is called \emph{holomorphic}. If $f_I$ is holomorphic for all imaginary units of $\mathbb K$, then
the function $f$ is called \emph{slice regular}.

  If $f$ is a slice regular function, then
  the {\em Cullen} or {\em slice derivative of} $f$ is defined as
  \[f'_c(x+Iy) = \frac{1}{2} \left( \frac{\partial}{\partial x} - I \frac{\partial}{\partial y} \right) f_I(x+yI) .\]

\end{definition}
 It turns out that $f'_c$ is a slice regular function
(see \cite{libroGSS})
  and from (\ref{sr}) one easily obtains that $f'_c=\dfrac{\partial f}{\partial x}$.

The property of being holomorphic along the slices $\Omega_I$ for all
imaginary units $I$ of $\mathbb K$, forces slice regular functions to
be affine along entire regions of each sphere of type $x+\bS y$.

In
fact, the local representation formula for quaternionic slice regular
functions on slice domains (see, e.g., \cite{local, slice domains}),
states that, if $L,M,N \in \bS$, with $M \neq N$, are such that
$x+Ly,x+My,x+Ny$ belong to a suitable open neighborhood $U$ of $x+Iy$ in
the $2$-sphere $x+\bS y$, then the local representation formula
\begin{eqnarray}\label{local rep formula}
f(x+Ly)
&=& (M-N)^{-1} \left[M f(x+My) - N f(x+Ny)\right]+\\
&+&L (M-N)^{-1} \left[f(x+My) - f(x+Ny)\right]\nonumber
\end{eqnarray}
holds
 and, for
 $y \ne 0$, the \emph{spherical derivative}  of $f$ is defined by
\begin{equation}\label{spherical}
f'_s(x+Iy):=y^{-1}(M-N)^{-1} \left[f(x+My) - f(x+Ny)\right].
\end{equation}
Moreover, $f'_s$ is constant in the same neighborhood $U$ of $x+Iy$ in
$x+\bS y$ (see, e.g., \cite[Definition 3.1]{slice domains}).  The
analog of this representation formula holds for octonionic slice
regular functions as well (see, e.g., \cite[Proposition 6]{Ghil-Per},
\cite[Formula (5)]{Ghil-Per-Stop}).  A subclass of the class of slice
regular functions on a slice domain $\Omega\subseteq \bK$ particularly
resembles the class of holomorphic functions of one complex variable.
These functions are defined as follows: a slice regular function $f:
\Omega \to \mathbb K$ is said to be {\em slice preserving} if, and
only if, for all imaginary units $I$ of $\bS \subset \mathbb{K}$, we
have that $f(\Omega_I)\subseteq \mathbb{C}_I$, (see \cite{cayley} for
the case of octonions).

In a while   we will
make use of the notion of stem  function, which was defined by Ghiloni-Perotti in \cite{Ghil-Per}, on a class of subsets of the complex plane $\C$.
\begin{definition}\label{sets}
    A subset $D$ of \ $\C=\R+i\R$ is said to be {\em symmetric} (in $\C$) if $\overline D=\{\overline z : z \in D\}$ coincides with $D$. The
 \emph{(axial) symmetrization} $\widetilde E$ of a subset $E$ of $\bK$ is defined by 
 $$\widetilde{E} = \{ x + I y:  x,y \in \R, I\in \bS, (x + \bS y) \cap E \neq \varnothing \}.$$  
 A subset $\Omega$ of $\bK$ is called
 {\em (axially) symmetric} (in $\bK$) if $\widetilde{\Omega} = \Omega.$
\end{definition}
 The following definition was given in \cite{Ghil-Per} in the general
 case of a real alternative algebra  endowed with an
 anti-involution (or $\mathbb{R}$). For the purpose of this paper, and for the sake of
 simplicity, we will restrict to the cases $\R, \H, \Oc$.

  \begin{definition}\label{stem}
  Let $A$ denote either $\R$ or $\bK$ and let  $A_{\C}: = A\otimes_{\mathbb{R}} \mathbb{C}$  be the complexification of $A$. Let us adopt the usual representation $$A_{\C}=\{ x+\iota y : x,y \in A\}$$ where $\iota^2=-1$.
Consider a symmetric domain $D$ of $\C$.

If a function $F : D \ra A_{\C}$ is \emph{complex intrinsic}, that is
if $F( z)=\overline{F(\overline z)}$ for all $z\in D$, then $F$ is
called an {\em $A$-stem function} (or {\em stem function}) on $D$.

If $F:D \to A_{\C}$ is a stem function expressed by
$$F(z)=F_1(z)+\iota F_2(z)$$
then the function $f: \widetilde D \to A$
$$f(x + Iy) = F_1(x + iy) + I F_2(x + iy)$$ is called the
 {\em slice function} induced by $F$.

\end{definition}

Slice regular functions on symmetric slice domains can all be induced by stem functions, as the following result states (see, e.g., \cite{Ghil-Per}).

\begin{proposition}
If a slice domain $\Omega$ in  $\mathbb{K}$
  is axially symmetric,
then any slice regular
function $f:\Omega \to \bK$ is induced by a holomorphic stem  function $F : D =
\Omega_i \ra \bK_{\C}$.
\end{proposition}
As we have seen, stem functions can be defined in symmetric open
subsets $E$ of $\C$ that do not necessarily intersect the real
axis. As a consequence, holomorphic stem functions induce special
slice functions, still called \emph{slice regular functions}, defined
on symmetric domains $\widetilde E$ of $\bK$ which do not necessarily
intersect the real axis, so generalizing the initial notion of slice
regularity to the class of so called \emph{product
  domains} (see e.g. \cite{GPS-TAMS} for the terminology and the seminal paper on stem functions \cite{Ghil-Per}).


In this paper we will be mainly concerned with slice regular functions
defined on slice domains of $\bK$, which  in principle can be dealt with avoiding
reference to stem functions. However, by admitting on the stage the
point of view of stem functions, some results may be easily extended to
the case of product domains; moreover, the generation of slice regular
functions through holomorphic stem functions is exactly the same for
the case of quaternions and octonions, and hence such an approach has
the advantage to provide a natural unified vision in the two different
environments, thus simplifying technicalities and
presentation.

\begin{remark}\emph{
\noindent With reference to the notations of Definition \ref{stem},  the following facts have been proven (see, e.g., \cite{Ghil-Per}):
\begin{enumerate}[(a)]
\item If $F:D \to A_{\C}$ is expressed by $F(z)=F_1(z)+\iota F_2(z)$, with $F_j:D\ra A$ for $j=1,2$, then
$F$ is complex intrinsic if and only if
\begin{equation}\label{intrinsic}
F_1(z)=F_1(\overline z)\quad \mbox{ and}\quad F_2(z)=-F_2(\overline z)\ \forall z\in D.
 \end{equation}
\item If we take $A = \R$ then the  slice function $f$ induced by the stem function $F$
  is a slice preserving
function (i.e. $f(\widetilde D_I) \subset \C_I$ $\forall I \in \bS$).
\item The local representation formula (\ref{local rep
  formula}) holds, by definition, for slice functions.
\item  If $F=F_1+\iota F_2$
        is a holomorphic stem function which induces the slice regular function $f$, then its (complex) derivative
        $F'=F_1'+\iota F_2'$ is also a holomorphic stem function
which induces the slice derivative $f'_c$ of $f$.
\item If $f$ is a slice function generated by the stem function $F$, then for $y \ne 0$,
  \begin{equation}\label{sph_d} f'_s(x+Iy)=y^{-1} F_2(x+iy).
  \end{equation}
\end{enumerate}
}

\end{remark}
For most of the remaining properties of slice regular functions that
will be directly used in the sequel we will mainly refer the reader to
\cite{libroGSS, cayley}. As for the main applications and developments
of this theory, the reader can consult \cite{angella-bisi, Tori, librodaniele2, GGS2, local, slice domains, camshaft, Gori-Vlacci}, and e.g.
\cite{GPS} for generalizations.

\section{Parameterized quaternionic and octonionic Riemann manifolds  }\label{sec:2}
Following the case of classical parameterized surfaces and
parameterized Riemann surfaces in $\R^N$, we will give  new
definitions, useful in the quaternionc and octonionic settings of slice regular functions. As customary, a differentiable map will be called \emph{an immersion} if its differential is injective at all points of the domain of definition.

\begin{definition} Let $n,N$ be natural numbers with $N\geq n$ and let $\Omega$ be a domain in $\R^n$. A $C^1$ immersion
  \[f:    \Omega  \to \R^N\]
will be called a \emph{conformal or isothermal map} if the matrix of the differential of $f$ is conformal, i.e., if it satisfies
\[
^tdf(x)df(x)=k(x)I_n
\]
for a (never vanishing $C^0$) function $k: \Omega\to \R$
\end{definition}

Recall that, if $\langle \ , \ \rangle$ denotes the Euclidean scalar product in
$\R^{\dim\bK}\cong \bK$,  then for $I\in \bS$ the symbol $\C_I^\perp$ will denote the orthogonal complement of the
slice $\C_I$.

\begin{definition}\label{slice conformality and Riemann manifolds}
Let $\Omega$ be a slice domain in 
 $\bK\cong \R^{\dim_{} \bK}$
and let $N\geq \dim_{} \bK$ be a natural number. 
Let $f:   \Omega  \to \R^N$ be a $C^1$  map (immersion).
If for  all $I\in \bS$ and all $x,y\in \R$  the differential $df(x+Iy)$ is such that both
\[
df(x+Iy)_{|\C_I}
\]
and
\[
df(x+Iy)_{|\C_I^\perp}
\]
are conformal, then $f$ will be called a \emph{slice conformal or slice isothermal map (immersion)}.

If  $f$ is an injective immersion, then it will be called a \emph{slice
  conformal or slice isothermal parameterization} and the
parameterized manifold $f(\Omega)$ in $\R^N$ will be called
a  \emph{(parameterized)  hypercomplex
    Riemann manifold of $\R^N$. In particular, when  $\bK = \H$ we
    refer to it as a quaternionic Riemann manifold and in the case $\bK = \Oc$ as a octonionic
    Riemann manifold.}

In case $f: \Omega \to \R^N$ itself is a conformal parameterization,
then the parameterized  hypercomplex Riemann manifold $f(\Omega)$ in $\R^N$ will be
called a \emph{special (parameterized)  hypercomplex Riemann  manifold} of
$\R^N$.
\end{definition}

The notion of parameterized quaternionic or octonionic Riemann manifold turns out to be quite
natural, as the significant examples that we will present show. To
construct the examples we will need a direct and easy method to
compute the differential of a $C^1$ immersion$f$ defined in
a slice domain $\Omega$ of $\bK\cong \R^{\dim \bK}$ and with values in $\R^N$.

\subsection{The standard set of curves and the case of the differential of a slice regular function}\label{four-curves}

For $I\in \bS$, let us consider a
point $x+Iy\in \C_I\subset \bK$ and choose $L \in\bS$ orthogonal to $I$.
In the same spirit of \cite[proof of Proposition 3.1]{Orient-preserv} and \cite[Proposition 3.1]{GPSoctonionic}, and with a similar purpose, we will use the following set of curves.
For $y\neq 0$ set
\begin{enumerate}
\item the curve $\alpha(t)=(x+t) +Iy$, such that $\alpha(0)=x+Iy$ and $\alpha'(0)=1$;\\
\item the curve $\beta_I(t)=x+I(y+t)$, such that $\beta_I(0)=x+Iy$ and $\beta'(0)=I$;\\
\item the curve $\Gamma_L(t)=x +\gamma(t)y$, where $\gamma(t)$ is an arc of a maximum circle $C_\gamma$ of $\bS$
  such that $\gamma(0)=I$ and that $\gamma'(0)=\frac Ly$; 
  hence $\Gamma_L(0)=x +Iy$ and $\Gamma'_L(0)=L$;
\end{enumerate}
Instead, when $y=0$ and  so
$x+Iy = x$,
the first curve   is
\begin{enumerate}
\item $\alpha(t)=x+t$, such that $\alpha(0)=x$ and $\alpha'(0)=1;$
\end{enumerate}
and the second two  coherently become:
\begin{enumerate}
\item[(2)-(3)] $\beta_I(t)=x+It, \beta_L(t)=x+Lt$,
  such that $\beta_I(0)=\beta_L(0)=x$ and $\beta_I'(0)=I,
  \beta_L'(0)=L.$\\
\end{enumerate}

In order to present the next definition we need to recall a well known fact: given any $I\in \bS \subset \bK$, then both in the case of quaternions and in the case of octonions, it is possible to complete $\{1, I\}$ to an orthonormal positively oriented standard basis
$$
\{1, I, I_2,\ldots,I_{\dim\bK-1}\}
$$
of the divison algebra $\bK$ (see, e.g., \cite{cayley} for the case of octonions).

\begin{definition} (Standard set of curves) For any $I\in \bS$, let us consider the
point $x+Iy\in \C_I\subset \bK$ and an orthonormal positively oriented standard basis  $\{1, I, I_2,\ldots,I_{\dim\bK-1}\}$ of the division algebra $\bK$.  
%
%
%
The {\em standard set of curves} at the point $x + Iy$ consists:
\begin{itemize}
\item  for $y \ne 0$, of the curves
$\{\alpha, \beta_I, \Gamma_{I_l},l = 2,\ldots, \dim \bK -1 \} $;
\item for $y = 0$,  of the curves
$\{\alpha, \beta_I, \beta_{I_l}, l = 2,\ldots,\dim \bK -1 \}.$
\end{itemize}
\end{definition}

We desire now to  use the standard set of curves to calculate the differential $df$ of $f$ and to point out some
of its features.  Indeed, when naturally used with a slice regular
function $f:\Omega \to \bK$, defined on a slice domain $\Omega$  of
$\bK$,  this set  of  curves reveals an easy tool to compute and directly
interpret the real differential $df(x+Iy): \R^{\dim \bK} \to \R^{\dim \bK}$ of the
function $f$. But its full use will be seen in the sequel of this
paper, in more general situations.

 To calculate $df,$  after fixing $I\in \bS$, 
a direct computation shows that
\begin{eqnarray*}
df(x+Iy)1
&=&df(x+Iy)\alpha'(0)=\frac{d}{dt}_{|_0}f(\alpha(t))
= \frac{d}{dt}_{|_0}f(x+t+Iy)\\
&=& \dfrac{\partial f}{\partial x}(x+Iy)
\end{eqnarray*}
Analogously, and since $f$ is slice regular,
\begin{eqnarray*}
df(x+Iy)I &=&df(x+Iy)\beta'(0)=\frac{d}{dt}_{|_0}f(\beta(t))
=\frac{d}{dt}_{|_0}f(x+I(y+t))\\
&=& \dfrac{\partial f}{\partial y}(x+Iy)
= I \dfrac{\partial f}{\partial x}(x+Iy)
\end{eqnarray*}
In particular we have that  $Idf(x+Iy)1=I\dfrac{\partial f}{\partial x}(x+Iy)
=\dfrac{\partial f}{\partial y}(x+Iy)=df(x+Iy)I$ and hence
\[
df(x+Iy)_{|\C_I} =
\begin{bmatrix} f'_c(x+Iy),If'_c(x+Iy) \end{bmatrix}.
\]
Therefore, the real differential
\[
df(x+Iy)_{|(\R+I\R)}: \R^2 \to \R^{\dim \bK}
\]
is a conformal matrix. Let us now continue. The local representation formulas
(\ref{local rep formula}) and (\ref{spherical}) yield, for any $L\in\bS$ such that
$L \perp I$
\begin{eqnarray*}
df(x+Iy)L&=&df(x+Iy)\Gamma_L'(0)=\frac{d}{dt}_{|_0}f(\Gamma_L(t)) \\
&=& \frac{d}{dt}_{|_0}(\gamma(t) (M-N)^{-1} \left[f(x+My) - f(x+Ny)\right])\\
&=&L y^{-1}(M-N)^{-1} \left[f(x+My) - f(x+Ny)\right]\\
&=&Lf'_s(x+Iy)
\end{eqnarray*}  
hence we have that
\[
df(x+Iy)_{|\C_I^\perp} =
\begin{bmatrix} I_2f'_s(x+Iy),\ldots,I_{\dim \bK-1}f'_s (x+Iy)\end{bmatrix}.
\]
Therefore, the real differential
\[
df(x+Iy)_{|(\R + I\R)^\perp}: \R^{\dim \bK-2} \to \R^{\dim \bK}
\]
is a conformal matrix as well.
 Notice that
even if both
$
df(x+Iy)_{|\C_I}
$
and
$
df(x+Iy)_{|\C_I^\perp}
$
are conformal, the full differential
$
df(x+Iy)
$
may not be conformal in general.

\subsection{The differential of a smooth slice function}
In this subsection we exhibit the connection between conformality
properties of a slice function defined in a symmetric slice domain $f: \Omega = \widetilde{D} \ra \bK$ and its
stem function $F: D \ra \bK.$
Since the local representation formula (\ref{local rep formula}) holds
for slice functions, then for $y \ne 0$ we obtain, as seen above, the identity $df(x + Iy) L
= L f'_s(x + Iy) = L y^{-1}F_2(x + iy)$  for every imaginary unit $L \bot\, I$; therefore, for $y \ne 0,$ the restriction of the
differential to the orthogonal complement of the slice $\C_I$  is conformal (if nonzero).

Assume now that $F \in \mathcal{C}^3(D).$ By Proposition 7(2)  in
\cite{Ghil-Per}, $f$ is $C^1$ and hence we can calculate the
differential $df(x + Iy)_{|\C_I}:$
\begin{eqnarray*}
  df(x + Iy) 1 &=&  \partial_x F(x + iy) = \partial_x(F_1(x + iy) + I F_2(x + iy)),\\
  df(x + Iy) I &=& \partial_y F(x + iy) = \partial_y(F_1(x + iy) + I F_2(x + iy)).
\end{eqnarray*}
Therefore, in terms of the stem function $F$, for $y \ne 0$ by
formula (\ref{sph_d}), the differential $df$
may be  written as 
\[
df(x+Iy)=
\begin{bmatrix}
  \tiny\partial_x F(x + iy) &\tiny{\partial_y F(x + iy)}& \frac{I_2F_2(x + iy)}{y} & 
\ldots   & \frac{I_{\dim\bK -1} F_2(x + iy)}{y}
\end{bmatrix}.
\]
Passing to the limit as $y\to 0$, then $\frac {F_2(x+Iy)}{y}$ tends to $\partial_y F_2(x)$ and so
\[df(x)=
\begin{bmatrix}
  \tiny\partial_x F(x) &\tiny{\partial_y F(x )}& I_2\partial_yF_2(x ) & 
\ldots
& I_{\dim\bK -1} \partial_y F_2(x)
\end{bmatrix}.
\]
Therefore, for $y \ne 0$, $df(x+Iy)_{|\C_I^{\bot}}$ is conformal if and only if $F_2(x + iy) \ne 0$ and $df(x)_{|\C_I^{\bot}}$ is conformal
if and only if $\partial_y F_2(x) \ne 0.$

In the case $F$ is holomorphic and $y\neq 0$, the corresponding formula becomes
\[
df=
\begin{bmatrix} f'_c& I f'_c& I_2f'_s&\ldots& I_{\dim\bK-1}f'_s
\end{bmatrix}
\]
and, when $y=0$, we have $f'_s=f'_c$ and so
\[
df=
\begin{bmatrix}
 f'_c& I f'_c & I_2f'_c& \ldots  & I_{\dim\bK-1}f'_c
\end{bmatrix}
\]
which implies that $df(x)$ is conformal if  $f'_c(x)\neq 0$.


Let's sum up these observations in the following

\begin{proposition}\label{slice_dif}
  Let $A$ denote either $\R$ or $\bK$, and let
  $f:\widetilde D \to A$ be a slice function generated by a $C^3$ stem function $F$ defined on a symmetric domain $D \subset \C = \R +i\R.$
Then $df(x+Iy)_{|\C_I^{\bot}}$  is conformal if nondegenerate. Moreover,
\begin{enumerate}[(a)]
\item if $dF$ is conformal on $D,$ then $f$ is
slice conformal on $\widetilde D$. In particular, if $F$ is
holomorphic, then $f$ is slice regular and hence a slice conformal immersion if
$df$ has full rank;
\item if $A = \R,$ $dF$ is conformal, $\partial_y F_2\ne 0$ on $\R \cap D$ and if $F_2 \ne 0$ on $D\setminus \R,$
  then $f$ is slice preserving  and slice conformal on $\widetilde D$.
\end{enumerate}
\end{proposition}

\begin{proof} We are left to consider only the case  $A = \R$. Since
$f$ is slice preserving, 
    then $df$ written with respect to the decomposition
    $\bK = {\C_I} \oplus {\C_I}^{\bot}$ is of the form
\[
   \begin{bmatrix}
         df(x+Iy)_{|\C_I}& 0\\
         0&df(x+Iy)_{|\C_I^{\bot}}
   \end{bmatrix}.
\]
\end{proof}

\begin{remark}\label{rem_dif}\emph{  Notice that if  a stem function $F$ is conformal, it is not necessarily holomorphic.
    In the case $A = \R, $ the stem function $F: D \ra A_{\C}$ is
    conformal if and only if $F$ is either holomorphic or
    antiholomorphic or, to put it differently, if and only if $df(x +
    Iy)_{|\C_I}$ is conformal on $\widetilde{ D}.$  Furthermore, notice that conformality of both $df(x+Iy)_{|\C_I}$ and $df(x+Iy)_{|\C_I^{\bot}}$ does not imply that $df$ has full rank.}
\end{remark}

\subsection{Slice conformal curves}

Using Proposition \ref{slice_dif}, we
can state the following result.

\begin{theo}\label{potential_result}
Let $A$ denote either $\R$  or $\bK.$
Let $D$ be a symmetric domain in $ \C = \R + i\R$ and $G,H:
{D} \ra A_{\C}$ be stem functions with  $G, H \in {\mathcal C}^3(D).$  Write
$G = G_1 +\iota G_2,H = H_1 + \iota H_2$  and let $F_1 = (G_1,H_1),F_2 = (G_2,H_2).$
Let $$f: \widetilde{D} \rightarrow \bK \times\bK$$ be  the slice curve
induced by the map $F = (G, H) =F_1 +\iota F_2: D \rightarrow A_{\C}^2$ in the following way
\begin{eqnarray*}
  f(x + Iy) &=& (G_1(x +iy)+IG_2(x+iy),H_1(x +iy)+IH_2(x+iy))\\
  & =: &(g(x +Iy), h(x+ Iy)).
\end{eqnarray*}
Assume that:
\begin{enumerate}[(a)]
\item the differential $dF$ is conformal on $D$;
\item the partial derivative $\partial_y F_2\ne 0$ on $\R \cap D$
  and $F_2 \ne 0$ on $D\setminus \R.$
\end{enumerate}
Then $df(x + Iy)|_{|\C_I}$ and $df(x + Iy)|_{|\C_I^{\bot}}$ are both conformal.
 If, in addition, $f$ is an injective immersion, then $f$  is a slice conformal parameterization of $f(\widetilde{D}).$

 In the case $A = \R,$ if $F$ is injective, then $F_2\ne 0$ on $D\setminus\R$ is automatically fulfilled;
 hence  if we assume that $F$ is injective, $dF$  is conformal on $D$ and  $\partial_y F_2\ne 0$ on $\R \cap D$,
 then $f$ is an injective immersion, and hence a slice conformal parameterization  of $f(\widetilde{D}).$

\end{theo}

\begin{proof}
To prove  the first part of the theorem, notice that
by Remark \ref{rem_dif}
the conformality of $dF$ implies $df_{|\C_I}$ conformal. The assumption $(b)$  and
Proposition \ref{slice_dif} imply that $df(x + Iy)_{|\C_I^{\bot}}$ is conformal.

We are left to prove that, if $A = \R$ and $F$ is injective, then $f$ is injective and $df$ has full rank.
Notice that
the non vanishing of $F_2$ off the real axis
follows from the    injectivity of $F$: indeed, at least one of the values $G_2(x + iy)$ or
    $H_2(x + iy)$ must be nonzero, otherwise $F(x + iy) = F(x - iy).$

Let  us first show that injectivity of $F$ implies the injectivity of $f$.
To this aim, consider $z = x + iy,w = u + iv \in D$ and assume that $f(x + Iy) = f(u + Jv).$  Then
\begin{eqnarray*}
&&G_1(x +iy)+IG_2(x+iy)=G_1(u +iv)+JG_2(u+iv),\\
&&H_1(x +iy)+IH_2(x+iy)= H_1(u +iv)+J H_2(u+iv).
\end{eqnarray*}
By assumption $G_l$ and $H_l,$ $l = 1,2$ are real valued and by (\ref{intrinsic}) the functions $G_2,H_2$ vanish at real points, so we have the following:
\begin{eqnarray*}
&&G_1(x +iy) = G_1(u +iv) = G_1(u - iv),\\
&& H_1(x+iy) = H_1(u + iv)=H_1(u - iv),\\
J=I: \,&& G_2(x + iy) = G_2(u+iv), \, H_2(x+iy) = H_2(u + iv),\\
J = -I: &&\,G_2(x + iy) = -G_2(u+iv), \, H_2(x+iy) = -H_2(u + iv),\\
J \ne \pm I:&& G_2(x + iy) = G_2(u+iv)= H_2(x+iy) = H_2(u + iv)=0.
\end{eqnarray*}
The injectivity of $F$ excludes the last possibility unless $y =v = 0.$ In this case
$F(x) = F(u),$ so $x = u.$ If $J = I$ then $F(x + iy) = F(u + iv)$ so $x + Iy = u + Iv.$ If $J = -I$ then
$G_2(x + iy) = - G_2(u+iv) = G_2(u - iv),$ $H_2(x + iy) = -H_2(u+iv) = H_2(u - iv).$ Because $G_2, H_2$ are even in $y$ we have
$F(x + iy) = F(u - iv)$ which implies that $x + iy = u - iv$ and hence $x + Iy = x + (-I)(-y),$ so $f$ is injective.

To see that the rank of $df$ is full, notice that on the real axis $\partial_yF_2$ does not vanish by   assumption, and $F_2$ does not vanish off the real axis.
Since both  $df_{|{\C_I}}$ and $df_{|{\C_I^{\bot}}}$ are conformal and
$df$ has the following block structure
\[
 df =\begin{bmatrix}df_{|{\C_I}} & df_{|{\C_I^{\bot}}}
 \end{bmatrix}=
\begin{bmatrix}
         dg_{|{\C_I}}& 0\\
         0 &              dg_{|{\C_I^{\bot}}}\\
         dh_{|{\C_I}}& 0\\
         0 &              dh_{|{\C_I^{\bot}}}
\end{bmatrix}
\]
the rank of $df$ is full.

\end{proof}

\begin{remark} \emph{ With reference to the preceding statement and proof, notice that conformality of $F$ does not imply conformality of $G$ and $H$.}
\end{remark}

  \begin{remark} \emph{A statement analogous to the
    one  of Theorem \ref{potential_result} holds in a ``$n$-vectorial'' version, i.e., for  maps  $F  : D \rightarrow A_{\C}^n$ defined by $n$-tuples of stem functions.}
\end{remark}

\subsection{Quaternionic and octonionic slice regular curves }\label{sec:curves}

We will use the standard notion of curve in the quaternionic and octonionic  setting.
\begin{definition}
Let $\Omega\subseteq \bK$ be a slice domain, and let $$f:\Omega \to \bK^2$$ $$f(q)=(g(q), h(q))$$ be a map whose components  $g, h: \Omega \to \bK$ are slice regular functions. If  $f$ is an immersion, then $f$ will be called \emph{a slice regular curve (in $\bK^2$)}.
\end{definition}
Let us now consider a slice regular curve $f:\Omega \to \bK^2$, with slice regular components $g, h:\Omega \to \bK$,  and choose any $I\in \bS.$
Using the standard set of curves defined in Subsection \ref{four-curves}, we get that the differential
\[
df : \R^{\dim\bK} \to \bK^2
\]
assumes the form
\[
df=
\begin{bmatrix}
 g'_c& I g'_c & I_2g'_s,\ldots,I_{\dim\bK-1 }g'_s\\
 h'_c& I h'_c & I_2h'_s,\ldots,I_{\dim\bK-1}h'_s
\end{bmatrix}.
\]
The first $2$ columns of this $(2\dim\bK) \times \dim\bK$ real matrix, and separately the last $(\dim\bK - 2)$ columns of the same matrix, are
 orthogonal to each other and with the same norms, and hence $F$ is slice isothermal. In conclusion we have proved

\begin{proposition}\label{slice curves}
Let $\Omega\subseteq \bK$ be a slice domain, and let $f:\Omega \to \bK^2$ be a slice regular curve. If $f$ is injective, then $f(\Omega)$
is a parameterized  hypercomplex   Riemann manifold in $\bK^2$, and  the map  $f: \Omega \to  f(\Omega)$   is a slice conformal parameterization. In particular,
graphs of slice regular curves are  parameterized  hypercomplex   Riemann manifolds in $\bK^2.$
\end{proposition}

As we already pointed out, in general $f$ is (a slice conformal but)
not a conformal parmeterization. It is well known in fact that the
slice regular functions $f, g$ are in general not conformal at non
real points of $\Omega$ (see, e.g., \cite{libroGSS}), and hence $f$
cannot be a conformal parameterization in general.\\

We end this section with a natural question, on how the quaternionic
or octonionic parameter can be changed between slice regular
quaternionic or octonionic curves having the same image. Indeed, let
us consider $\Omega, \Omega' \subseteq \bK$ slice domains, $f
=(f_1,f_2) : \Omega\to \bK^2$ and $g =(g_1,g_2) : \Omega' \to \bK^2$
injective, slice regular curves with the same image
$f(\Omega)=g(\Omega')$. In this situation, we may assume that locally
$g_1$ is injective.  Then the local equalities $f_1(q) = g_1(q')$ and
$f_2(q) = g_2(q')$ imply $$f_2 = g_2 \circ (g_1^{-1} \circ f_1)$$ and
since $f_2, g_2:\Omega' \to \bK$ are slice regular functions, this
functional equation is in general not valid.  Nevertheless, we know
that it holds if, for instance, $g_1^{-1} \circ f_1: \Omega \to
\Omega' $ is a slice preserving regular function. Hence, we can make
the following

\begin{remark}\label{change par}
 \emph{ Let $f$ and $g$ be injective immersions having the same
  image $\Gamma\subseteq \bK^2$. If a change of quaternionic or
  octonionic parameter between $f$ and $g$ is a slice preserving
  invertible function, then $f$ is a slice conformal parameterization if, and only if, $g$ is a slice conformal parameterization.}
\end{remark}

What established in this section can be directly reformulated for the case of  slice regular curves $f: \Omega \to \bK^n$ defined on slice domains $\Omega \subseteq \bK.$  To conclude, we point out that Remark \ref{change par} is valid in a more general setting, as explained in the next remark.
\begin{remark}
\emph{Let $f$ be a slice isothermal parameterization having the  hypercomplex Riemann manifold  $\Gamma\subseteq \R^N$ as its image. Then, for every regular slice preserving invertible change of parameter $\phi$ between slice domains, the map $f\circ\phi$ is a slice isothermal parameterization 
for the hypercomplex  Riemann manifold $\Gamma$.}
\end{remark}

The following remark should better explain the definition of slice conformal immersion that has been adopted.

\begin{remark}\emph{
Let $\{1,i,j,k\}$ be the standard basis of $\H$, and let $f: \H \to \H^2$ be the function $$f(x + Iy) = (x +Iy, x+\psi(I)y)$$ where
 $\psi : \bS \rightarrow \bS$ is the odd $C^\infty$ function defined by
 \[
 \psi(\alpha i+\beta j+\gamma k)= \dfrac{\alpha^3i+\beta j+\gamma^3k}{\sqrt{\alpha^6+\beta^2+\gamma^6}}
 \]
i.e., when $\langle \ , \ \rangle$ denotes the Euclidean scalar product of $\R^4\cong \H$, by
 \[
 \psi(I)= \dfrac{\langle I , i \rangle^3i+\langle I , j \rangle j+\langle I , k \rangle^3k}{\sqrt{\langle I , i \rangle^6+\langle I , j \rangle^2+\langle I , k \rangle^6}}
 \]
While applying the standard set of curves, take the point $x+Iy=x+iy$ (i.e., $I=i$) with  $y\neq 0$, choose $J=j$ and use the curves
\begin{eqnarray*}
&&\Gamma_j(t)=i\cos (t/y)+j \sin (t/y)=i\exp(-k (t/y)), \\
&&\Gamma_k(t)=i\cos (t/y)+k \sin (t/y)=i\exp(j (t/y))
\end{eqnarray*}
Direct computations show that
\begin{eqnarray*}
df(x+iy)1 &=& (1,1), \,\quad  df(x+iy)i = (i, \psi(i)) = (i,i)\\
df(x+iy)j
&=&  \frac{d}{dt}_{|_0}\left(x+ i\exp(-k (t/y)) y, x+  \frac{i\cos^3 (t/y)+j \sin (t/y)}{\sqrt{\cos^6 (t/y)+ \sin^2 (t/y)}}y\right)\\
&=&(j, j)\\
df(x+iy)k
&=&  \frac{d}{dt}_{|_0}\left(x+ i\exp(j (t/y)) y, x+  \frac{i\cos^3 (t/y)+j \sin^3 (t/y)}{\sqrt{\cos^6 (t/y)+ \sin^6 (t/y)}}y\right)\\
&=&(k, 0)
\end{eqnarray*}
Thus:
\[
df(x+iy)=\begin{bmatrix}
1&0 & 0 & 0\\
0 & 1& 0 & 0\\
0 & 0 & 1  & 0 \\
0 & 0 & 0 &1\\
1 & 0 & 0 & 0\\
0 &1& 0 & 0\\
0 &0& 1& 0\\
0 &0 & 0  &0\\
\end{bmatrix}
\]
and hence the last two columns have different norms. In conclusion such an $f$ is not a slice isothermal parameterization.
}
\end{remark}
 \section{Other examples of hypercomplex Riemann manifolds}
 \label{4}

\subsection{The Riemann sphere} This example generalizes to real dimensions $4$ and $8$ the case of the Riemann sphere in the complex setting.
\begin{proposition}
 Let us set $m = \dim \bK\in \{4, 8\}$. Consider the unit sphere $S^{m}\subset \R^{m+1}\cong \bK\times \R$ and the inverse of the stereographic projection from the north pole $N=(0,\ldots,0,1)$ of $S^{m}$ onto the equatorial plane $\bK\cong \R^{m}$, namely
\begin{eqnarray}
f: \R^{m}\cong \bK  \to  S^{m}\setminus \{N\} \subset \bK\times \R \cong \R^{m+1},
\end{eqnarray}
defined by
\begin{eqnarray}
  f(x+Iy)=\left(\dfrac{2(x+Iy)}{1+x^2+y^2}, \dfrac{-1+x^2+y^2}{1+x^2+y^2} \right).
\end{eqnarray}
Then $S^{m}\setminus \{N\}$ is a special parameterized hypercomplex  Riemann manifold and  the map $f$ is a  conformal parameterization.
Analogous statement can be proved for the stereographic projection from the south pole $S.$
\end{proposition}

\begin{proof} It is well known that the inverse of stereographic projection is conformal, the rest follows.
\end{proof}


We can now conclude by exhibiting the ``Riemann" structures of $1$-dimensional quaternionic manifolds of the spheres   $S^4 \subset \R^5$ and $S^8 \subset \R^9$. In the case of $S^4$, this structure corresponds to that of slice  quaternionic manifold, as defined in \cite{GGS}. In the case of $S^8$, it corresponds to a natural generalization to the case of octonions.

\begin{theo} Let us set $m = \dim \bK\in \{4, 8\}$. Let $f$ and $h$ be the following maps
\begin{eqnarray*}
f: \R^m\cong \bK  \to  S^m\setminus \{N\} \subset \bK \times \R \cong \R^{m+1}
\end{eqnarray*}
\begin{eqnarray*}
  f(x+Iy)=\left(\dfrac{2(x+Iy)}{1+x^2+y^2}, \dfrac{-1+x^2+y^2}{1+x^2+y^2} \right)
\end{eqnarray*}
and
\begin{eqnarray*}
h: \R^m\cong \bK  \to  S^m\setminus \{S\} \subset \bK\times \R \cong \R^{m+1}
\end{eqnarray*}
\begin{eqnarray*}
  h(x+Iy)=\left(\dfrac{2(x-Iy)}{1+x^2+y^2}, \dfrac{1-x^2-y^2}{1+x^2+y^2} \right).
\end{eqnarray*}
Then the differentiable conformal atlas $\left\{(\bK, f), (\bK, h)\right\}$ endows $S^m\subset \R^{m+1}$ with a structure of slice quaternionic or slice octonionic manifold.

\end{theo}
\begin{proof}
A direct computation shows that the transition map
\[
h^{-1}\circ f = \overline{g^{-1}\circ f}:\bK\setminus \{0\} \to \bK\setminus \{0\}
\]
has the form
\[
(h^{-1}\circ f)(q)= \overline{(g^{-1}\circ f)(q)}=\frac{\bar q}{q^2}=\frac{1}{q}
\]
and hence it is a slice regular and slice preserving function.
\end{proof}

\subsection{The   helicoidal hypercomplex  manifold}\label{1.2} This further example generalizes to quaternions and octonions the case of the helicoid in the complex setting,  whose classical isothermal parameterization is given by
$
g: \C\cong \R^2 \to \R^3\cong \C\times \R
$
defined as
$g(x+iy)=(\sinh x \cos y +i\sinh x \sin y, y).$
\begin{proposition} \label{Elicoidale}
Let  the map
\begin{eqnarray*}
f: \bK \to \bK \times \Im(\bK)
\end{eqnarray*}
be defined by
\begin{eqnarray*}
f(x+Iy)=(\sinh x \cos y +I\sinh x \sin y, Iy)
\end{eqnarray*}
for $I\in \bS$, $x, y\in \R$. Then $f(\bK)$ is a parameterized hypercomplex Riemann
manifold (diffeomorphic to $\bK$) and $f$ is a slice isothermal parameterization. This manifold will be called \emph{quaternionic helicoidal manifold} if $\bK = \H$ or
\emph{octonionic helicoidal manifold} if $\bK = \Oc,$ and denoted by $\mathscr E$.
\end{proposition}

\begin{proof}
The map $f$ is induced by the stem map
$$F = (G,H): \C \to( \R+\iota \R)^2,$$
$$G(x + iy) = \sinh x (\cos y +\iota \sin y),\, H(x + iy) =  \iota y$$
whose components are those of the classical conformal parametrization of the helicoid. We need to check that the assumptions of Theorem \ref{potential_result} hold.
The injectivity of $F$ is obvious since the last component is injective in $y$ and the first is injective in $x;$  moreover, $H_2(x + iy)  =  y \ne 0$ on $\C\setminus\R$ and $\partial_y H_{2}(x)=1 \ne 0$ on $\R$. Now, since as we said $dF$ is conformal, then by Theorem \ref{potential_result},
 the map $f$ is a slice conformal parameterization, and the proof is complete.
\end{proof}

It may be interesting to see how the use of the standard set of curves leads to the explicit calculation of the differential of the slice isothermal parameterization of the helicoidal manifold
$$f(x+Iy)=(\sinh x \cos y +I\sinh x \sin y, Iy).$$
For a fixed $x+Iy\in\C_I$, let us compute
$df(x + Iy)1$ and $df(x + Iy)I:$
\begin{eqnarray*}
df(x+Iy)1
&=& \frac{d}{dt}_{|_0}(\sinh (x+t) \cos y +I\sinh (x+t) \sin y, Iy)\\
&=& (\cosh x \cos y +I\cosh x \sin y, 0)\\
df(x+Iy)I
&=& \frac{d}{dt}_{|_0}(\sinh x \cos (y + t)+I\sinh x \sin (y+t), I(y+t))\\
&=& (-\sinh x \sin y +I\sinh x \cos y, I).
\end{eqnarray*}
Moreover,  from Proposition \ref{slice_dif} we know that for  $l = 2,\ldots,
\dim \bK-1$  \[df(x+Iy)I_l =\left(I_l\frac{\sinh x \sin
    y}{y},I_l\right).\]

\noindent In the case $\bK = \H$, if  we  set
$$\H \ni x_1+x_2I+x_3J+x_4K\cong (x_1, x_2, x_3, x_4) \in \R^4$$
and
\begin{eqnarray*}
\H\times \Im(\H)&\ni &(x_1+x_2I+x_3J+x_4K, y_2I+y_3J+y_4K)\\
&\cong& (x_1, x_2, x_3, x_4, y_2, y_3, y_4)\in \R^7 ,
\end{eqnarray*}
then, for $y\neq 0$, we get
\[
df(x+Iy)=\begin{bmatrix}
\cosh x\cos y & -\sinh x \sin y & 0 & 0\\
\cosh x \sin y &  \sinh x \cos y & 0 & 0\\
0 & 0 & \frac{\sinh x \sin y}{y} & 0 \\
0 & 0 & 0 &\frac{\sinh x \sin y}{y}\\
0 & 1 & 0 & 0\\
0 &0 & 1 & 0\\
0 & 0 & 0 & 1
\end{bmatrix}
\]
and for $y=0$, we have, taking the limit (and coherently with the use of the standard curves):
\[
df(x)=\begin{bmatrix}
\cosh x&0 & 0 & 0\\
0 &  \sinh x & 0 & 0\\
0 & 0 & \sinh x  & 0 \\
0 & 0 & 0 &\sinh x\\
0 & 1 & 0 & 0\\
0 &0 & 1 & 0\\
0 & 0 & 0 & 1
\end{bmatrix}.
\]
As expected, $df(x+Iy)$ is slice conformal and $df(x)$ is conformal. The case $\bK=\Oc$ is completely analogous.

\subsection{The catenoidal hypercomplex manifold}
The case of the catenoid in the complex setting, parameterized by the conformal map
$$g: \C\cong \R^2 \to \R^3\cong \C\times \R$$
defined by
$$g(x+iy)=(\cosh x \cos y +i\cosh x \sin y, x)$$
generalizes to  quaternions and  octonions   as well.
\begin{proposition} \label{catenoidale}
Let  the map
\begin{eqnarray*}
f:\R\times \bS(-\pi, \pi) \to \bK\times \R \cong \R^{\dim \bK + 1}
\end{eqnarray*}
be defined by
\begin{eqnarray*}
f(x+Iy)=(\cosh x \cos y +I\cosh x \sin y, x)
\end{eqnarray*}
Then $f(\R\times \bS(-\pi, \pi))$ is a parameterized
hypercomplex Riemann manifold and $f$ is a slice isothermal parameterization. This manifold will be called
\emph{quaternionic catenoidal manifold} if $\bK = \H$ or \emph{octonionic catenoidal manifold} if $\bK = \Oc$. 
\end{proposition}

\begin{proof}

The map $f$ is induced by the stem map
$$F = (G,H): \R\times (-i\pi, i\pi) \to( \R+\iota \R)^2,$$
\[G(x + iy) = \cosh x (\cos y +\iota \sin y), H(x + iy) = x,\]
whose components are those of the classical conformal parametrization of the catenoid.

Obviously the map $F$  is injective on $\R \times (-i\pi,i\pi)$, and
$G_2(x,y) \neq 0$ outside the real axis; moreover the derivative $\partial_y G_{2}(x,0) = \cosh x$ never vanishes. Since $dF$ is conformal, then by Theorem \ref{potential_result},  the map $f$ is a slice conformal parameterization. This completes the proof.
\end{proof}

Again, it may be interesting to explicitly present the differential of the slice isothermal parameterization of the catenoidal manifold
$$f(x+Iy)=(\cosh x \cos y +I\cosh x \sin y, x)$$
which may be computed by means of the standard set of curves.
In the case $\bK = \H,$ if we set
$$\H \ni x_1+x_2I+x_3J+x_4K\cong (x_1, x_2, x_3, x_4) \in \R^4$$
and
$$\H\times \R \ni (x_1+x_2I+x_3J+x_4K, y_1)
\cong (x_1, x_2, x_3, x_4, y_1)\in \R^5$$
then, for $y\neq0$, we have
\[
df(x+Iy)=\begin{bmatrix}
\sinh x\cos y & -\cosh x \sin y & 0 & 0\\
\sinh x \sin y &  \cosh x \cos y & 0 & 0\\
0 & 0 & \frac{\cosh x \sin y}{y} & 0 \\
0 & 0 & 0 &\frac{\cosh x \sin y}{y}\\
1 & 0 & 0 & 0
\end{bmatrix}.
\]
Moreover, for $y=0$, we  coherently obtain:
\[
df(x)=\begin{bmatrix}
\sinh x&0 & 0 & 0\\
0 &  \cosh x & 0 & 0\\
0 & 0 & \cosh x  & 0 \\
0 & 0 & 0 &\cosh x\\
1 & 0 & 0 & 0
\end{bmatrix}.
\]
Again, $df(x+Iy)$ is slice conformal and $df(x)$ is conformal. In the case of octonions we obtain similar matrices.

As in the real case, once both naturally  embedded in $\mathbb{K}^2,$ the  catenoidal hypercomplex manifold can be transformed to a part of an 
helicoidal hypercomplex manifold through a family of parameterized hypercomplex Riemann manifolds.\\

Let the part of the helicoidal manifold embedded in $\mathbb{\bK}^2$ be parameterized by $h : \R\times \bS(-\pi, \pi) \rightarrow \bK^2,$ induced by the stem map $$H(x + iy) := (\sinh x (\cos y +\iota \sin y), \iota y)$$ and the embedded catenoidal manifold parameterized by $c : \R\times \bS(-\pi, \pi) \to \bK^2,$ induced by the stem map $$C(x + iy) := (\cosh x(\cos y +\iota \sin y), x).$$
We claim that
\[
H_{\theta} := H \cos\theta + C \sin\theta,\,  \theta \in [0,\pi/2]
\]
defines a family of conformal injective immersions with
$H_0 = H,$ $H_{\pi/2 } = C.$

The  differential
$dH_{\theta}: \C \to \bK \times\bK \cong  \R^{2\dim\bK}$
is given by

\[
dH_{\theta}(x+iy)=\begin{bmatrix}
A \cos y & - B\sin y \\
A\sin y &  B \cos y \\
0 & 0   \\
\vdots&\vdots\\
0 & 0   \\
\sin \theta & 0\\
0 & \cos \theta \\
0 & 0   \\
\vdots&\vdots\\
0 & 0
\end{bmatrix},
\]
where $A = (\cosh x \cos \theta + \sinh x \sin \theta)$ and $B =
(\cosh x \sin \theta + \sinh x \cos \theta).$ It is obvious that the
columns are orthogonal  to each other, and a direct computation shows that 
their norms are equal.  If we
write $H_{\theta} = (F_{\theta},G_{\theta}),$ then $G_{\theta}(x + iy)
= x\sin \theta +\iota y \cos \theta.$ If $x\sin \theta +\iota y \cos
\theta = u\sin \theta + \iota v \cos \theta,$ either $\theta = 0$ (and
then we have the stem function for the helicoidal manifold, for
which the injectivity has already been proved), or $x =u.$ Then either
$\theta = \pi/2$ (in which case we have the stem map for the catenoidal
surface) or else $\iota y = \iota v $ so $x + iy = u + iv.$  As a consequence, $G_{\theta, 2}$ is injective for all $\theta \in (0,\pi /2).$ Because
$\partial_y G_{\theta,2}(x)= \cos \theta \ne 0$  on the real axis for all $\theta \in
(0,\pi/2),$ all the conditions of Theorem
\ref{potential_result} are fulfilled and hence $H_{\theta}$ induces a
family of slice conformal injective immersions.

\section{The hypercomplex logarithm and $n$-th root}
\label{sec:log}

\subsection{The hypercomplex logarithm}

To define the complex logarithm one usually uses either the helicoid or the graph of the exponential function. Since we have shown that the latter in case of $\bK$ is a hypercomplex manifold, the logarithm can be defined using the projection on the second coordinate (compare Remark \ref{remark log}).

We will show here how  the helicoidal  hypercomplex manifold defined in the previous section can be adapted to be the natural domain for the definition of a quaternionic logarithm.  Compared to the logarithm defined by the graph of exponential function, this definition facilitates the identification of the argument and is therefore easier to use in the constructions which include continuations of the logarithm.

\begin{proposition} Let  $f: \bK \to \bK\times \Im(\bK)$ be the map
defined by
\begin{eqnarray*}
f(x+Iy)=(\sinh x \cos y +I\sinh x \sin y, Iy)
\end{eqnarray*}
for $I\in \bS$, $x, y\in \R$. Let $\bK^+=\{q\in \bK : \Re\ q >0\}$, and set $\mathscr E_\bK^+:=f(\bK^+)$.
\noindent The  $\mathscr E_\bK^+$-\,exponential map 
\[
E : \bK \to \mathscr E_\bK^+ \subset \bK \times \Im(\bK)
\]
defined by:
\[
E(x+Iy) = (\exp (x + Iy), Iy) = (\exp x \cos y + I\exp x \sin y, Iy)
\]
is an immersion and a diffeomorphism between $\bK$ and $\mathscr E_\bK^+.$ In the case of quaternions, it endows $\mathscr E_\H^+$ with a structure of slice quaternionic manifold (see, e.g., \cite{GGS}), which is different from the structure of hypercomplex Riemann manifold defined in Proposition \ref{Elicoidale}. However, this manifold will be denoted simply by $\mathscr E_\bK^+$, and called \emph{the logarithm manifold}.

\end{proposition}
\begin{proof}
The proof replicates part of the one of Proposition \ref{Elicoidale}.
\end{proof}

\begin{remark} \emph{ (a) 
If $\pi : \bK \times \Im(\bK) \to \bK$ denotes the projection on the first factor, then by definition the following equality holds
\[
(\pi \circ E)(q)= \exp(q)
\]
for all $q\in \bK$.\\
\noindent (b) Unlike what happens in the complex setting, the map $\pi :\mathscr E_\bK^+ \to \bK$ is not a covering. It is not an open map as well, due to the fact that $\exp :\bK \to \bK$ is not an open map (it has  a  non--empty degenerate set consisting of  spheres).}
\end{remark}


We will now define the $\mathscr E_\bK^+$-\,logarithm on  $\mathscr E_\bK^+$ and exhibit some of its properties.

\begin{definition} Let $\mathscr E_\bK^+$ be the logarithm manifold.
The $\mathscr E_\bK^+$-\,logarithm
\[
L : \mathscr E_\bK^+ \subset \bK \times \Im(\bK)   \to \bK
\]
 is defined as follows, in terms of the real logarithm  $\log$,
\[
L (q, p) := \log |q| + p,
\]
where $p$ is called the \emph{argument} of $q$, and denoted by $\Arg(q)$:  hence $$L (q, p) := \log |q| + \Arg(q).$$
Indeed, if $(q,p)\in \ \mathscr E_\bK^+$, then $q=r\exp p$ for $r=|q|$ and $L$ can be rewritten as
\[
L (r\exp p, p) = \log r + p. 
\]
\end{definition}
 The following result and definition explain why the logarithm manifold  is a natural domain of definition for the $\mathscr E_\bK^+$-\,logarithm. Indeed,  this hypercomplex manifold plays the role of an ``adapted" blow-up of $\bK$ at points $x\in \R$ with $x\neq 0$.
 
\begin{proposition}
The map
\[
L : \mathscr E_\bK^+ \to \bK
\]
is the inverse of the $\mathscr E_\bK^+$-\,exponential $E$, and a diffeomorphism from  $\mathscr E_\bK^+$ to $\bK$.
\end{proposition}
\begin{proof}

Let us read the $\mathscr E_\bK^+$-\,logarithm through the parameterization
\[
E(x+Iy)=(\exp x \cos y +I\exp x \sin y, Iy)
\]
of $\mathscr E_\bK^+$. By composition we get that $L \circ E$ becomes the identity map of $\bK$
\[
x+Iy \mapsto (\exp x (\cos y +I \sin y), Iy) \mapsto \log (\exp x) + Iy = x+Iy
\]
Analogously, $E \circ L $ becomes the identity map of $\scrE_\bK^+$
\[
(r\exp p, p) \mapsto \log r + p \mapsto (\exp (\log r)\exp p, p) \mapsto (r\exp p, p).
\]
The assertion is now proved.
\end{proof}
As a consequence,  in the case of quaternions, the map $L$ is a slice regular map from  the logarithm manifold $\mathscr E_\H^+$ to $\H$, with respect to the structure of slice regular manifold induced by $E$ on $\mathscr E_\H^+$ (see, e.g., \cite{GGS}). We point out that the definition of the $\mathscr E_\bK^+$-\,logarithm $L$ is not referred to the structure of  helicoidal Riemann manifold defined on $\mathscr E_\bK$ in Proposition \ref{Elicoidale}.
\begin{definition}\label{det1log}
Let
$
\pi :  \mathscr E_\bK^+  \subset \bK \times \Im(\bK)  \to \bK\setminus \{0\}
$
denote the natural projection
\[
(q, p) \mapsto q
\]
and let $\Omega \subset \mathscr E_\bK^+$ be  a path connected subset such that $\pi_{|_{\Omega}}$ is injective. Then, the map 
\[
\log_\bK : \pi(\Omega) \to \bK
\]
defined by
\[
\log_\bK q = L( \pi_{|_{\Omega}}^{-1}(q) )
\]
is called a \emph{branch or a determination of  the hypercomplex logarithm} on $\pi(\Omega)$. 
\end{definition}
Notice that, as expected, with the notations of  Definition \ref{det1log} we have that
\[
\exp (\log_\bK q)= \pi (E  (L( \pi_{|_{\Omega}}^{-1}(q) )))=\pi (\pi_{|_{\Omega}}^{-1}(q) )=q
\]
for all $q$ in $\pi(\Omega)$

\begin{remark}\emph{
It is important to notice that, unlike what happens in the case of the complex logarithm, and with the exception of the principal branch  (see, e.g., \cite[Definition 3.4]{Gen-Vig}),  no continuous branch of the hypercomplex logarithm can be defined on any open set $A\subset \bK\setminus \{0\}$ which contains a strictly positive real point  $x_0$, and hence a small segment $(x_0-\epsilon, x_0+\epsilon) \subset \R^+$. Indeed, for any $I\in \bS$, on each slice $A_I$, the branches of the hypercomplex logarithm coincide with those of the complex logarithm of the slice $\C_I$. As a consequence, there is no choice of $J\in \bS$ along $(x_0-\epsilon, x_0+\epsilon) \subset \R^+$ which can make a (non principal) branch of the hypercomplex logarithm a continuous function.\\
\indent On the other hand, if $A\subset \C_I\setminus \{0\} \subset \bK\setminus \{0\}$ is simply connected, any continuous branch of the hypercomplex logarithm
along $A$ coincides with the appropriate branch of the complex logarithm along $A$. 
In particular, this happens when  $\alpha: [-a, a]
\to \C_{I}\setminus \{0\} \subset \bK\setminus \{0\}$ is a
continuous curve having its image in a small disc $\Delta$ centered at
a non zero real point $x$ with $\Delta \subset \C_{I}\setminus\{0\}$, and such that
$\alpha(0)=x$. 
 We will
address this issue in a forthcoming paper.  
}
\end{remark}


We conclude this section by pointing out a different possible definition of the hypercomplex Riemann manifold on which to define  the hypercomplex logarithm.

\begin{remark}\label{remark log} \emph{
The definition of a hypercomplex logarithm could be given, alternatively, using the graph of the exponential function
\[
\Gamma(\exp) = \{ (q, \exp q) : q\in \bK\}
\]
which has a natural structure of hypercomplex Riemann manifold (see
Subsection \ref{sec:curves}), with the function $f(q)=(q, \exp q)$ as
a slice isothermal parameterization. Indeed the logarithm could be 
defined as the slice regular function from the ``reversed" graph $
\Lambda(\exp) = \{ (\exp w, w) : w\in \bK\} $ onto $\bK$, coinciding
with the projection onto the second factor. The advantage of the
approach that we actually adopted in this paper stays also in that it calls
into scenery the helicoidal and logarithm manifolds, which
more closely follow the path of the complex setting.
  }
\end{remark}

\subsection{The hypercomplex  $n$-th root}\label{sec:root}
To give a proper definition of the $n$-th root function over the quaternions and octonions, we will first of all define a suitable hypercomplex Riemann manifold, which will be useful to find a possible domain for such a function.
\begin{proposition} \label{radice} Let $n\in \mathbb N$, with $n>1$,  and
let  the map
\[
f: \R\times \bS(-\pi n,\pi n) \to \bK\times \bK\cong \R^{2\dim \bK}
\]
be defined by
\[
f(x+Iy)=(\sinh x \cos y +I\sinh x \sin y, n\exp(I\frac{y}{n}))
\]
for $I\in \bS$, $x ,y\in \R$. Then $f(\R\times \bS(-\pi n,\pi n))$ is a parameterized Riemann
hypercomplex manifold (diffeomorphic to $\R\times \bS(-\pi n,\pi n)$) and $f$ is a slice isothermal parameterization. This manifold will be denoted by $\mathscr Q_\bK (n)$.
\end{proposition}
%

\begin{proof}
The map $f=(g, h) $ is induced by the stem map
\begin{eqnarray*}
F &=& (G,H) :\R\times (-i\pi n, i\pi n) \to( \R+\iota \R)^2,\\
G(x+iy) &=& \sinh x \cos y +\iota \sinh x \sin y, \, H(x + iy) = n\exp(\iota \frac{y}{n}))
\end{eqnarray*}
whose components are those of the classical conformal parameterization of the Riemann surface of the $n$-th root.
The map $F$ is $C^\infty$ and injective: indeed $H(x+iy)=H(u+iv)$ implies
\[
\exp(\iota \frac{y}{n})=\exp(\iota\frac{v}{n})
\]
whence $\frac{y-v}{n}=2\pi m$ for some integer $m$. Hence $y-v=2\pi nm$ implies $y=v$. Since $G$ is injective in $x,$ we now deduce $x=u$. The injectivity of $F$ is then proved. Because $H_2(x + iy) = n \sin \frac{y}{n},$ we have $\partial_y H_{2}(x) = 1 \ne 0.$

Since, as we said, $dF$ is conformal, then by  Theorem \ref{potential_result} the map $f$ is a slice conformal parameterization, and the proof is complete.
\end{proof}

Again, it is of interest to explicitly compute the differential of the slice conformal parameterization
$$f(x+Iy)=(\sinh x \cos y +I\sinh x \sin y, n\exp(I\frac{y}{n})).
$$
Since the first component of $f$ has already been analyzed in Subsection \ref{1.2}, we will only  compute the differential of  the function $n\exp(I\frac{y}{n}):$

\begin{eqnarray*}
dh(x+Iy)1&=& \frac{d}{dt}_{|_0}(n\exp(I\frac{y}{n}))= 0\\
dh(x+Iy)I &=& \frac{d}{dt}_{|_0}n\exp(I\frac{y+t}{n})\\
&=&  -\sin (\frac{y}{n})+I\cos(\frac{y}{n}).
\end{eqnarray*}

\noindent In the case $\bK = \H,$ if we set
$$\H \ni x_1+x_2I+x_3J+x_4K\cong (x_1, x_2, x_3, x_4)\in \R^4$$
and
\begin{eqnarray*} \H\times \H &\ni& (x_1+x_2I+x_3J+x_4K, y_1+ y_2I+y_3J+y_4K)\\
&\cong& (x_1, x_2, x_3, x_4, y_1, y_2, y_3, y_4)\in\R^8
\end{eqnarray*}
then, for $y\neq 0$, we have
\[
df(x+Iy)=\begin{bmatrix}
\cosh x\cos y & -\sinh x \sin y & 0 & 0\\
\cosh x \sin y &  \sinh x \cos y & 0 & 0\\
0 & 0 & \frac{ \sinh x \sin y}{y} & 0 \\
0 & 0 & 0& \frac{ \sinh x \sin y}{y}\\
0 & -\sin (\frac{y}{n})& 0 & 0\\
0 &\cos (\frac{y}{n})& 0  & 0\\
0 & 0 & \frac{n\sin (y/n)}{y}  & 0\\
0 & 0 & 0 & \frac{n\sin (y/n)}{y}\\
\end{bmatrix}
\]
and for $y=0$, we coherently obtain:
\[
df(x)=\begin{bmatrix}
\cosh x&0 & 0 & 0\\
0 &  \sinh x & 0 & 0\\
0 & 0 & \sinh x  & 0 \\
0 & 0 & 0 &\sinh x\\
0 & 1 & 0 & 0\\
0 &0 & 1 & 0\\
0 & 0 & 0 & 1
\end{bmatrix}.
\]
As expected, $df(x+Iy)$ is slice conformal and $df(x)$ is conformal. The situation in the case $\bK=\Oc$ is totally analogous.
\vskip .2cm

We will now see how to use $\mathscr Q_\bK(n)$ to construct an appropriate domain for the quaternionic or octonionic $n$-th root function.

\begin{proposition}\label{different structure} Let $f(x+Iy)=(\sinh x \cos y +I\sinh x \sin y, n\exp(I\frac{y}{n}))$ be as in Proposition \ref{radice}, and let us set $$\mathscr Q_\bK^+(n):=f(\R^+\times \bS(-\pi n,\pi n)).$$
The map
\[
\phi_n : \R^+\times \bS(-\pi n,\pi n)) \to \mathscr Q_\bK^+(n)
\]
defined by
\begin{eqnarray*}
&&\phi_n(x+Iy) = (\exp (x + Iy), n\exp(I\frac{y}{n}))
\end{eqnarray*}
is an injective immersion and a diffeomorphism between $\R^+\times \bS(-\pi n,\pi n) $ and $\mathscr Q_\bK^+(n)$.   Indeed, in the case of quaternions, $\phi_n$  defines on $\mathscr Q_\H^+(n)$ a structure of slice regular manifold (see \cite{GGS}) different from the one induced by the contruction of Proposition \ref{radice}. However, this manifold will be denoted simply by $\mathscr Q_\bK^+(n)$, and called \emph{the $n$-th root manifold}.
\end{proposition}
\begin{proof}
The proof replicates the one used to establish Proposition \ref{radice}
\end{proof}

We will now define the hypercomplex $n$-th root on the $n$-th root manifold,  and establish some of its properties.

\begin{definition}Let $n\in \mathbb N$, with $n>1$ and let $\mathscr Q_\bK^+(n)$ be the  $n$-th root manifold.
The  $n$-th root
\[
R_n : \bK \times \bK \supset \mathscr Q_\bK^+(n) \to \bK
\]
is defined as follows, for all $r\in \R^+$ and $p\in \bS (-\pi n, \pi n)$:
\[
R_n(r\exp p, n\exp(\frac{p}{n}) )=\sqrt[n]{r} \exp(\frac{p}{n})
\]
or directly (and equivalently), for all $(q, s)\in  \mathscr Q_\bK^+(n)$, by
\[
R_n(q, s) =\sqrt[n]{|q|}\ \frac{s}{n}\,.
\]
Indeed, this last formulation of the definition extends in a natural fashion, to $\overline{\mathscr Q_\bK^+(n)}=f((\R^+\cup \{0\})\times \bS[-\pi n,\pi n ])$ as
\[
R_n(0, s)=0
\]
and
\[
R_n(r, -n)=-\sqrt[n] r
\]
for all $s\in nS^3$ and all $r\geq 0$.
\end{definition}
As stated in Proposition \ref{different structure}, and analogously to what happens in the case of the logaritm, the definition of the $n$-th root function is not referred  to the structure of hypercomplex Riemann manifold defined on $\mathscr Q_\bK(n)$ in Proposition \ref{radice}. Indeed, the structure that is naturally involved with the $n$-th root functions is the one defined in Proposition \ref{different structure}.

As it clearly appears, there is natural space and interest for the study of differential geometry of hypercomplex Riemann  manifolds and, in particular, for the study of their curvature, of their mean curvature and minimality. This will be the subject of a forthcoming paper.


\begin{thebibliography}{99}
\bibitem{angella-bisi} D. Angella, C. Bisi, {\em Slice-Quaternionic Hopf Surfaces}, J. Geom. Anal. {\bf 29} (2019),1837-1858. https://doi.org/10.1007/s12220-018-0064-9
\bibitem{Tori} C. Bisi, G. Gentili, {\em On quaternionic tori and their moduli space}, J. Noncommutative Geom. {\bf 12} (2018), 473– 510. https://doi.org/10.4171/JNCG/284
\bibitem{librodaniele2} F.~Colombo, I.~Sabadini, and D.~C. Struppa, Noncommutative functional calculus. Theory and applications of. slice hyperholomorphic functions, vol. 289 of Progress in Mathematics. Birkh{\"a}user/Springer Basel AG, Basel, 2011.
\bibitem{GGS} G. Gentili, A. Gori, G. Sarfatti, {\em A direct approach to quaternionic manifolds}, Math. Nachr., {\bf 290} (2017), 321-331. https://doi.org/10.1002/mana.201500489
\bibitem{GGS2} G. Gentili, A. Gori, G. Sarfatti, {\em On Compact Affine Quaternionic Curves and Surfaces}, J. Geom. Anal. {\bf 31} (2021), 1073-1092. https://doi.org/10.1007/s12220-019-00311-2
\bibitem{local} G. Gentili, C. Stoppato, {\em  A local representation formula for quaternionic slice regular functions},
  Proc. Amer. Math. Soc. {\bf 149} (2021), 2025–2034. https://doi.org/10.1090/proc/15339
\bibitem{slice domains} G. Gentili, C. Stoppato, {\em Geometric function theory over quaternionic slice domains},
  J. Math. Anal. Appl., {\bf 495} (2021), 1-38.
\bibitem{libroGSS} G. Gentili, C.  Stoppato, D. C. Struppa, Regular Functions of a Quaternionic Variable, Springer Monographs in Mathematics, Springer, Berlin-Heidelberg, 2013.
\bibitem{cayley} G. Gentili, D. C. Struppa, {\em Regular functions on the space of Cayley numbers}, Rocky Mountain J. Math., {\bf 40} (2010), 225-241.
\bibitem{Gen-Vig} G. Gentili, I. Vignozzi, {\em The Weierstrass factorization theorem for slice regular functions over the quaternions},
  Ann. Global Anal. Geom. {\bf 40} (2011), 435-466. https://doi.org/10.1007/s10455-011-9266-0
\bibitem{Ghil-Per} R. Ghiloni, A. Perotti, {\em Slice regular functions on real alternative algebras},
Adv. Math.,
{\bf 226} (2011),1662-1691. https://doi.org/10.1016/j.aim.2010.08.015.
\bibitem{camshaft} R. Ghiloni, A. Perotti, {\em Zeros of regular functions of quaternionic and octonionic variable: a division lemma and the camshaft effect}, Ann. Mat. Pura Appl. (4), {\bf 190} (2011), 539-551.
\bibitem{Orient-preserv} R. Ghiloni, A. Perotti, {\em On a Class of Orientation-Preserving Maps of $\mathbb R^4$}, J. Geom. Anal. {\bf 31}  (2021), 2383-2415. https://doi.org/10.1007/s12220-020-00356-8.
\bibitem{Ghil-Per-Stop} R. Ghiloni, A. Perotti, C. Stoppato, {\em Singularities of slice regular functions over real alternative $*$-algebras},
Adv. Math., {\bf 305} (2017), 1085-1130, https://doi.org/10.1016/j.aim.2016.10.009.
\bibitem{GPS-TAMS}  R. Ghiloni, A. Perotti, C. Stoppato, {\em The algebra of slice functions}, Trans. Amer. Math. Soc. {\bf 369} (2017), 4725-4762.
\bibitem{GPS} R. Ghiloni, A. Perotti, C. Stoppato, (2020). {\em Division algebras of slice functions}, Proc. Roy. Soc. Edinburgh Sect. A, {\bf 150} (2020), 2055-2082.
\bibitem{GPSoctonionic} R. Ghiloni, A. Perotti, C. Stoppato, {\em Slice regular functions and orthogonal complex structures over $\R^8$}, Preprint (2021), arXiv:2109.12902 [math.CV].
\bibitem{Gori-Vlacci} A. Gori, F. Vlacci, {\em On a criterion of local invertibility and conformality for slice regular quaternionic functions}, Proc.
  Edimb. Math. Soc.  {\bf 62} (2019), 97-105.
\bibitem{Gurl} K. G\"urlebeck, K. Habetha, W. Spr\"ossig,  Holomorphic Functions in the Plane and n-dimensional
Space, Birkh\"auser Verlag, Basel (2008)
\bibitem{Riemann} B. Riemann, {\em Fondamenti di una teorica generale delle funzioni di una variabile complessa},
 Ann. Mat. Pura Appl. {\bf 2}, (1859), 288-304.
\end{thebibliography}
\end{document}